\documentclass[12pt, leqno, letterpaper]{article}

\usepackage{ifxetex}

\usepackage[letterpaper]{geometry}

\ifxetex
  \usepackage[T1]{fontenc}
  \usepackage[no-math]{fontspec}

  \usepackage[mtphrd,mtpfrak,slantedGreek,noamssymbols,
  	      subscriptcorrection,zswash]{mtpro2}

  \DeclareMathSizes{12}{12}{9}{7}
\else
  \usepackage[T1]{fontenc}
  \usepackage{lmodern}
  \newcommand{\mbf}{\mathbf} 

  \usepackage{mathrsfs}

\fi
  
\usepackage{amssymb}
\usepackage[scr=boondox, scrscaled=1]{mathalfa}

\usepackage{amsmath}
\usepackage{graphicx}
\usepackage{enumitem}
\usepackage{comment}
\usepackage{url}
\usepackage{marginnote}
\usepackage{booktabs}
\usepackage{bm}
\usepackage{mdframed}
\usepackage{xfrac}
\allowdisplaybreaks

\usepackage{float}
\floatstyle{ruled}
\newfloat{algorithm}{hpt}{lop}
\floatname{algorithm}{Algorithm}

\usepackage{amsthm}
\newtheoremstyle{mytheoremstyle}
{\topsep}                    
{\topsep}                    
{\slshape}                   
{0.32in}                           
{\bfseries}                  
{.}                          
{0.08in}                       
{}  

\theoremstyle{mytheoremstyle}

\makeatletter
\@ifclassloaded{report}{
   \newtheorem{theorem}{Theorem} }
{  \newtheorem{theorem}{Theorem}[section] }
\makeatother

\newtheorem*{theorem*}{Lemma}

\newtheorem{lemma}[theorem]{Lemma}
\newtheorem{corollary}[theorem]{Corollary}

\newtheorem{remark}[theorem]{Remark}

\newtheorem{assumption}[theorem]{Assumption}

\newtheorem*{defn}{Definition}

\renewenvironment{proof}[1][\proofname]{{\scshape #1. }}
{\qed\vspace{0.16in}}

\usepackage{titlesec}
\usepackage{marginnote}

\titleformat{\section}[block]
{\centering\large\bfseries\scshape}
{\thesection.}{0.08in}{}[]
\titlespacing{\section}{0in}{0.16in}{0.16in}

\titleformat{\subsection}[runin]
{\bfseries} 
{\thesubsection.}{0.04in}{}[.]
\titlespacing{\subsection}{0.32in}{0.08in}{0.08in}

\renewcommand{\abstractname}{\bf Abstract}

\renewenvironment{abstract}
{\small
\begin{center}
  \bfseries \abstractname\vspace{0in}\vspace{0pt}
\end{center}
\list{}{
	\setlength{\leftmargin}{0.5in}
    \setlength{\rightmargin}{\leftmargin}
	\itemindent \parindent} 
  \item\relax}
{\endlist}

\usepackage{caption}
\usepackage{hyperref}
\hypersetup{colorlinks=true, allcolors=deluge}

\usepackage{natbib}[mincitenames=4]

\newcommand{\ci}{\citet*}

\global\mdfdefinestyle{clean}{%
linecolor=black,linewidth=0.5pt,%
leftmargin=0.32in,rightmargin=0.32in
}

\usepackage{color}
\usepackage[table]{xcolor}

\definecolor{red}{rgb}{1,0,0}
\definecolor{gray}{rgb}{0.5,0.5,0.5}
\definecolor{darkgray}{rgb}{0.4,0.4,0.4}
\definecolor{blue}{rgb}{0,0,1}
\definecolor{green}{rgb}{0,1,0}

\definecolor{deluge}{RGB}{124, 113, 173}
\definecolor{bamboo}{RGB}{220, 92, 5}
\definecolor{yellow}{RGB}{255, 172, 0}
\definecolor{orange}{RGB}{255, 144, 0}
\definecolor{oyster}{RGB}{151, 139, 125}
\definecolor{coral}{RGB}{199, 186, 167}
\definecolor{downy}{RGB}{110, 197, 184}

\newcommand{\diag}{\textbf{diag}}

\usepackage{pgffor}

\foreach \x in {A,B,...,Z,a,b,...,z} 
{
  \expandafter\xdef\csname cal\x\endcsname{\noexpand 
	\ensuremath{\noexpand\mathcal{\x}}}
  \expandafter\xdef\csname scr\x\endcsname{\noexpand 
	\ensuremath{\noexpand\mathscr{\x}}}
  \expandafter\xdef\csname bb\x\endcsname{\noexpand 
	\ensuremath{\noexpand\mathbb{\x}}}
  \expandafter\xdef\csname rm\x\endcsname{\noexpand 
	\ensuremath{\noexpand\mathrm{\x}}}
  \expandafter\xdef\csname bf\x\endcsname{\noexpand 
	\ensuremath{\noexpand\mbf{\x}}}
}

\ifxetex
\let\gamma\upgamma
\let\epsilon\upepsilon
\let\tau\uptau
\let\pi\uppi
\let\kappa\upkappa
\let\omega\upomega
\fi

\newcommand{\E}{\calE}

\newcommand{\Y}{Y}

\usepackage{xcolor}

\title{{\Large \textbf{\textsc{Recovery Limits for Eigenvector Alignment in HDLSS Spiked Covariance Models}}}}

\author{
Hubeyb Gurdogan\footnote{Dept. of Mathematics,
University of California, Los Angeles, CA
({\tt hgurdogan@math.ucla.edu}).}
\hspace{0.32in} and \hspace{0.32in}
Alex Shkolnik\footnote{Dept. of Statistics
and Applied Probability, University of California,
Santa Barbara, CA and the Consortium for Data Analytics in Risk,
University of California, Berkeley, CA ({\tt shkolnik@ucsb.edu}).}
}

\date{{\normalsize This version: September 18, 2026.}}

\begin{document}

\ifxetex
  \let\lsum\sum
  \renewcommand{\sum}{\bm{\lsum}}

  \let\lprod\prod
  \renewcommand{\prod}{\bm{\lprod}}
\else
\fi

\maketitle
\thispagestyle{empty}
\pagenumbering{arabic}

\begin{abstract}
In a spiked covariance model, let \(\mathcal H\) and \(\mathcal B\) be selected
sample and population eigenvector matrices.  Their alignment matrix
\(\Xi=\mathcal H^\top\mathcal B\) is a latent oracle diagnostic of
direction-specific correspondence beyond subspace overlap.  We study its
uniform pathwise recovery in a nested high-dimensional, low-sample-size
experiment with fixed sample size and growing dimension.  At every fixed
model, an observable spectral contrast consistently estimates the row Gram
matrix \(\Xi\Xi^\top\).  For two or more spikes, however, no measurable
estimator sequence achieves uniformly reliable almost-sure path recovery of
\(\Xi\) over a family having common spike strengths and a common signal
subspace but different ordered eigenframes.  For every
\(0<\epsilon<1/\sqrt2\), a finite orbit forces each estimator to have some
orbit point at which the approximate path-recovery probability is at most a
variational bound \(u(\epsilon)\), with
\(\lim\limits_{\epsilon\downarrow0}u(\epsilon)=0\).  The worst-model
exact-recovery probability over a countable orbit is zero.
Thus the Gram summary is recoverable model by model, while the full alignment
matrix is not uniformly recoverable along the nested path.
\end{abstract}

\newpage
\setcounter{page}{1}

\section{Introduction}
\label{sec:intro}
Large covariance estimation is a central problem in high-dimensional
statistics and multivariate analysis.  Spiked covariance models give one of the
standard mathematical abstractions: a small number of leading eigenvalues
separate from a high-dimensional background, and the corresponding population
eigenvectors describe the dominant directions of variation.  This framework
appears in factor analysis, principal component analysis, signal processing,
and covariance estimation; see, for example, \ci{johnstone2001}, \ci{paul2007},
\ci{yao2015}, \ci{fan2008}, \ci{fan2016a}, \ci{wang2017}, and \ci{lam2020}.  In
approximate factor models, under pervasiveness, principal-component methods
recover the loading space, while the ordered covariance eigenvectors provide a
canonical spectral frame within that space.  Assessing alignment with the
individual ordered eigenvectors is therefore a finer question than assessing
recovery of their span
\citep{fan2013,fan2016b,bai2013}.

Let \(\mathcal B\) denote the selected population eigenvector matrix and let
\(\mathcal H\) denote the selected sample eigenvector matrix.  Under the
deterministic convention specified below, these are the fixed objects
\(\mathcal B=\nu(\Sigma)\) and \(\mathcal H=\nu(S)\), not bases chosen after
the fact for their respective column spaces.  The matrix
\[
  \Xi=\mathcal H^\top \mathcal B
\]
records the alignment between sample and population directions.  Its diagonal
entries measure how much each sample direction aligns with its corresponding
population direction, while its off-diagonal entries measure mixing across
different spike directions.  This object contains finer information than a
subspace distance or a canonical-angle summary.  The Gram matrix
\(\Xi\Xi^\top=\mathcal H^\top\mathcal B\mathcal B^\top\mathcal H\) sums over
the selected population eigenvectors and therefore retains only their
aggregate contribution to each pair of sample eigenvectors.  It does not
retain the individual entries of \(\Xi\).  Thus \(\Xi\) records the full
ordered pairwise correspondence between the selected sample and population
eigenvectors, whereas \(\Xi\Xi^\top\) is a coarser alignment summary.
Because the selected population eigenframe is unknown, \(\Xi\) is latent and
can be viewed as an oracle diagnostic of direction-specific PCA error.  We ask
whether it can nevertheless be recovered from the observed data in the HDLSS
regime, where \(n\) is fixed and \(p\) tends to infinity.

The HDLSS literature shows that PCA has behavior qualitatively different from
the proportional high-dimensional regime.  Foundational geometric results go
back to \ci{hall2005} and \ci{ahn2007}; consistency, strong inconsistency, subspace
consistency, and boundary behavior were developed in \ci{jung2009pca},
\ci{jung2012}, \ci{yata2009}, \ci{yata2013}, and \ci{shen2016}, with a broad
survey in \ci{aoshima2018}.  Related proportional-asymptotic results describe
phase transitions and limiting sample--population eigenvector overlap
\citep{paul2007}.  These works establish asymptotic geometry and consistency
properties of PCA and its modifications.  Those results do not by themselves
settle uniform recovery of the full realized alignment matrix \(\Xi\) in the
fixed-sample model family studied here.

In the proportional regime, \citet{lin2026eigenvector} show that generalized
eigenvector overlaps converge to deterministic counterparts, with rates, when
both dimensions grow.  Here the sample size is fixed and the target is the
realized random cross-alignment matrix rather than a deterministic equivalent.
In the fixed-sample single-spike setting, \citet{goldberg2022dispersion} recover
the squared overlap by a spectral contrast and then recover the signed scalar
alignment under a sign convention, whereas our obstruction uses multi-spike
orientation freedom.  \citet{jung2022} studies correction of rotation and
scaling bias in HDLSS principal-component scores under additional assumptions,
a related but different inferential target.  Structured sparsity can restore
HDLSS principal-direction consistency \citep{shen2013}; the orthogonal orbit
studied here is unstructured and does not preserve direction-specific sparsity.

The classical non-identifiability of factor models arises because, for every
\(Q\in O(q)\),
\[
  Y=LF^\top+E=(LQ)(FQ)^\top+E.
\]
Thus \(L\) and \(F\) are not separately identified from the observed data
without additional restrictions; see \ci{shapiro1985} and \ci{bai2013}.  This ambiguity does
not affect the quantities studied here.  Because the transformation leaves
\(Y\) unchanged, it also leaves
\[
  S=\frac1nYY^\top,
  \qquad
  \Sigma=\mathbb E[S]
\]
unchanged.  Consequently, the selected population eigenvector matrix
\(\mathcal B=\nu_{p\times q}(\Sigma)\) is invariant under the choice of
factor rotation \(Q\).  Since \(\Sigma\) is determined by the data law and
\(\nu_{p\times q}\) is a fixed deterministic measurable map, \(\mathcal B\)
is an identified population quantity.  The selected sample eigenvector
matrix \(\mathcal H=\nu_{p\times q}(S)\), and hence the alignment matrix
\(\Xi=\mathcal H^\top\mathcal B\), are likewise unchanged by the factor
rotation.  Therefore \(\Xi\) survives this classical non-identifiability:
our result concerns estimation of a well-defined eigenvector-alignment
quantity from the observed data in the fixed-sample HDLSS regime.

In the critical spike regime considered here, the leading population
eigenvalues grow linearly with \(p\),
so the sample eigenvectors are neither classically consistent nor completely
orthogonal to the population eigenspace.  For the family indexed by
\(\scrO\in O(q)\) in Section~\ref{sec:imp}, we first construct an observable
spectral-contrast matrix \(\Psi_{p,\scrO}\) from the eigenvalues of
\(S_{p,\scrO}\) and prove that
\[
  \lim\limits_{p\uparrow\infty}
  \|\Xi_{p,\scrO}\Xi_{p,\scrO}^\top-\Psi_{p,\scrO}\|_2=0
  \qquad\text{almost surely}.
\]
This convergence is established in Theorem~\ref{thm: quadopt}.
The smallest eigenvalue of this benchmark defines the limiting alignment scale
\[
  \psi_{\min}:=
  \lim\limits_{p\uparrow\infty}
  \lambda_{\min}(\Psi_{p,\scrO})^{1/2}>0
  \qquad\text{almost surely}.
\]
This scale is random and common to the model family.  We then show that the
full \(\Xi_{p,\scrO}\) does not admit uniformly reliable almost-sure path
recovery over the model family.
More precisely, for every \(0<\epsilon<1/\sqrt2\), there is a finite set
\(\mathbb G_\epsilon\subset O(q)\) such that every measurable estimator
sequence \(f=(f_p)\) satisfies
\begin{equation}
  \inf_{\scrO\in\mathbb G_\epsilon}
  \mathbb P\!\left(
  \limsup\limits_{p\uparrow\infty}
  \|\Xi_{p,\scrO}-f_p(\Y_{p,\scrO})\|_2
  \leq\epsilon\psi_{\min}
  \right)
  \leq u(\epsilon),
  \label{eq:intro-recovery-bound-v4}
\end{equation}
where \(\lim\limits_{\epsilon\downarrow0}u(\epsilon)=0\), as established in
Theorem~\ref{thm: approximate-recovery-bound}.  Over a countable orbit,
the corresponding infimum for exact recovery is zero
(Corollary~\ref{thm:anglesgeneralized}).

Corollary~\ref{cor:population-eigenspace-bound} also rules out uniformly
reliable recovery of the selected population eigenvector matrix
\(\mathcal B_{p,\scrO}\) by any measurable estimator, since such recovery
would induce an estimator of \(\Xi_{p,\scrO}\) by left multiplication with
the observed \(\mathcal H_{p,\scrO}^{\top}\).  Over one fixed countable set
\(\mathbb G_0\), every estimator sequence has infimum zero for the probability
of asymptotically exact recovery, even when every component of the model
family except \(\scrO\) is known.  Thus the common population signal subspace
is known, while its ordered eigenframe remains to be recovered; the conclusion
therefore goes beyond inconsistency of sample PCA. Related estimator-independent limitations were established by \citet{wahl2022lower}, who derives nonasymptotic lower bounds for Gaussian
covariance eigenprojection estimation under a prescribed spectrum. When bounded away from zero, these bounds preclude uniformly vanishing estimation error.  Thus, the quantification over all measurable estimators is not by itself the distinguishing feature of the Corollary~\ref{cor:population-eigenspace-bound}. Rather, it rules out uniformly reliable asymptotically exact recovery of the deterministically selected ordered eigenvector matrix $\mathcal B_{p,\scrO}$ along the entire dimension-indexed sequence, with fixed sample size, over a fixed countable family allowing non-Gaussian scores and noise, rather than bounding expected eigenprojection loss in a Gaussian experiment.

The paper therefore establishes a separation between two levels of alignment
information: at every fixed model the data recover alignment strength, encoded
by \(\Xi\Xi^\top\), but no single estimator achieves uniformly reliable
almost-sure path recovery of the full ordered cross-alignment matrix \(\Xi\)
over the orbit.

The rest of the paper is organized as follows.  Section~\ref{sec:imp}
introduces the orbit-indexed HDLSS model, states the spectral contrast result,
and gives the main approximate and exact path-recovery bounds.  The appendix
gives the precise eigenvector-selection convention, proofs of the main
theorems, and the auxiliary lemmas used in those proofs.

\section{Model and Main Results}
\label{sec:imp}
\subsection{The orbit-indexed HDLSS model}

Fix integers \(n>q\geq2\)  For each dimension
\(p\geq q\), let
\(\scrB_p\in\mathbb R^{p\times q}\) satisfy
\[
  \scrB_p^\top \scrB_p=pI_q .
\]
Thus the columns of \(p^{-1/2}\scrB_p\) are orthonormal, while the signal
loadings themselves have pervasive size \(\sqrt p\).  Fix
\(\scrW\in O(q)\) and fix
\(\Lambda=\diag(\lambda_1,\ldots,\lambda_q)\), where
\[
  \lambda_1>\lambda_2>\cdots>\lambda_q>0 .
\]
Let \(X\in\mathbb R^{n\times q}\) be random and let
\(\E_\infty\in\mathbb R^{\mathbb Z_+\times n}\) be an infinite noise array.
For each \(p\geq q\), \(\E_p\in\mathbb R^{p\times n}\) denotes the first
\(p\) rows of \(\E_\infty\).  Define
\[
  M:=X\scrW\Lambda\in\mathbb R^{n\times q}.
\]
For a set \(\mathbb G\subset O(q)\), the orbit-indexed data family is
\begin{align}
\mathcal D(\mathbb G)
:=
\left\{
(\Y_{p,\scrO})_{p\geq q}:
\Y_{p,\scrO}=\scrB_p\scrO M^\top+\E_p,\ \scrO\in\mathbb G
\right\}.
\label{eq:datafamily-v4}
\end{align}
The only free parameter in the family is \(\scrO\).  Because the spike
strengths are distinct, varying \(\scrO\) generally changes the population
covariance matrix; the orbit therefore indexes distinct statistical models,
not alternative bases for one model.

All random quantities are defined on one probability space carrying
\((M,\E_\infty)\), whose probability measure is denoted by \(\mathbb P\).
The index \(\scrO\in O(q)\) is deterministic.  For each fixed \(\scrO\),
the observation path \((\Y_{p,\scrO})_{p\geq q}\) is a function of these
common random coordinates; probabilities of recovery are evaluated separately
for each such fixed index.

\begin{assumption}\label{asm:asymp}
The random matrix \(X\) and the noise array \(\E_\infty\) satisfy the
following conditions.
\begin{enumerate}[label=\textup{(\alph*)}, itemsep=0.0in]
\item For some \(\delta>0\), \(X\) and \(\E_\infty\) are independent,
have mean zero, and
\[
  \mathbb E(X^\top X)/n=I_q,
  \qquad
  \mathbb E(\E_p\E_p^\top)/n=\delta^2I_p
  \quad\text{for every }p\geq q .
\]
\label{asm:meancov}
\item The random matrix \(M=X\scrW\Lambda\) has a positive continuous
density \(g\) on \(\mathbb R^{n\times q}\).
\label{asm:nonzeropdf}
\item Almost surely,
\[
  \lim\limits_{p\uparrow\infty}\frac1p\E_p^\top\E_p=\delta^2 I_n,
  \qquad
  \lim\limits_{p\uparrow\infty}\frac1p\scrB_p^\top\E_p=0 .
\]
\label{asm:SLLN}
\end{enumerate}
\end{assumption}

Let
\[
  S_{p,\scrO}:=\frac1n\Y_{p,\scrO}\Y_{p,\scrO}^\top
\]
be the sample second moment matrix.  Under Assumption~\ref{asm:asymp}(a),
the corresponding population second moment is
\begin{align}
  \Sigma_{p,\scrO}
  &:=
  \mathbb E[S_{p,\scrO}]
   =
  \scrB_p\scrO\Lambda^2\scrO^\top\scrB_p^\top+\delta^2I_p .
  \label{eq:expectationofS-v4}
\end{align}
The leading \(q\) eigenvalues of \(\Sigma_{p,\scrO}\) are
\(p\lambda_1^2+\delta^2,\ldots,p\lambda_q^2+\delta^2\), and are simple.
Thus the model is a spiked covariance model in which the spike eigenvalues
grow linearly with \(p\).

\subsection{Eigenvector selection}

For each \(m\geq q\) and \(\ell\geq1\), we fix once and for all a Borel
measurable selector
\[
  \nu_{m\times q}:\mathbb R^{m\times \ell}\to\mathbb R^{m\times q}
\]
which returns an orthonormal basis of the left singular subspace associated
with the largest \(q\) singular values.  For a symmetric positive
semidefinite matrix, the same notation denotes the selected leading
eigenvectors.  On the event that the relevant singular values are simple,
\(\nu_{m\times q}\) is just the usual ordered singular-vector map with a
fixed deterministic sign convention.  Its definition at multiplicities is
needed only to make all estimator maps measurable; no argument below depends
on arbitrary choices off the simple-spectrum event.  Appendix~\ref{sec:appendix}
records the precise convention used in the proof.
On the eventual simple-spectrum event, any two ordered measurable selectors
differ only through diagonal sign matrices. Define
\[
  \mathcal H_{p,\scrO}:=\nu_{p\times q}(S_{p,\scrO}),
  \qquad
  \mathcal B_{p,\scrO}:=\nu_{p\times q}(\Sigma_{p,\scrO}).
\]
Here \(\scrB_p\) is the fixed base loading matrix, whereas
\(\mathcal B_{p,\scrO}\) is the selected orthonormal population eigenframe.
The alignment matrix is
\begin{align}
  \Xi_{p,\scrO}
  :=
  \mathcal H_{p,\scrO}^{\top}\mathcal B_{p,\scrO}
  \in\mathbb R^{q\times q}.
  \label{eq:alignment-v4}
\end{align}
Its \((i,j)\) entry is the inner product between the \(i\)th selected sample
eigenvector and the \(j\)th selected population eigenvector.

\subsection{What is estimable}

The full alignment matrix contains \(q^2\) ordered pairwise overlaps.  Its
Gram-type summary
\[
  \Xi_{p,\scrO}\Xi_{p,\scrO}^\top
  =
  \mathcal H_{p,\scrO}^{\top}
  \mathcal B_{p,\scrO}\mathcal B_{p,\scrO}^{\top}
  \mathcal H_{p,\scrO}
\]
sums over the selected population eigenvectors and retains only their
aggregate contribution to the sample eigenvectors.  This coarser summary is
recoverable almost surely along the nested path from the sample spectrum.

\begin{defn}[Squared spectral contrast]\label{def:spectral_contrast_matrix}
Let \(S\in\mathbb R^{p\times p}\) be positive semidefinite.  Let
\(s_1^2\geq s_2^2\geq\cdots\) be its nonzero eigenvalues and let
\(n_+\) be the number of nonzero eigenvalues.  For every positive integer
\(k\), define the squared spectral contrast
\(C^2(S,k)\in\mathbb R^{k\times k}\) by
\begin{equation}
  C^2(S,k)
  :=
  \begin{cases}
  I_k
  -
  \dfrac{\sum_{j>k}s_j^2}{n_+-k}
  \diag(s_1^{-2},\ldots,s_k^{-2}),
  & n_+>k,\\[6pt]
  0_k,
  & n_+\leq k.
  \end{cases}
  \label{eq:def-spectral-contrast-v4}
\end{equation}
\end{defn}

The next theorem provides the estimable benchmark for our main negative
result.  A version of this spectral-contrast limit was derived in our earlier
unpublished manuscript on quadratic optimization bias
\citep{gurdogan2024quadratic}.  We restate and prove it here, in the notation
of the present model, so that the separation between recovery of the row Gram
matrix \(\Xi_{p,\scrO}\Xi_{p,\scrO}^{\top}\) and non-recovery of the full
alignment matrix \(\Xi_{p,\scrO}\) is entirely self-contained.

\begin{theorem}[Spectral contrast for the alignment Gram matrix]\label{thm: quadopt}
Under Assumption~\ref{asm:asymp}, let
\[
  \Psi_{p,\scrO}:=C^2(S_{p,\scrO},q).
\]
Then, for every fixed \(\scrO\in O(q)\),
\begin{equation}
  \lim\limits_{p\uparrow\infty}
  \left\|
  \Xi_{p,\scrO}\Xi_{p,\scrO}^\top-\Psi_{p,\scrO}
  \right\|_2=0
  \qquad\text{almost surely.}
  \label{eq:gram-estimable-v4}
\end{equation}
Moreover, each diagonal entry of \(\Psi_{p,\scrO}\) eventually lies in
\((0,1)\) almost surely.
\end{theorem}

Theorem~\ref{thm: quadopt} identifies exactly what the observable spectral
contrast recovers: the row Gram matrix of the model-specific target
\(\Xi_{p,\scrO}\).  This determines aggregate alignment strengths and row
inner products, but not the individual entries of \(\Xi_{p,\scrO}\).  The main
result shows that no measurable estimator can achieve uniformly reliable
almost-sure path recovery of the remaining ordered pairwise alignment
information over the orbit of distinct covariance models.  The proof of
Theorem~\ref{thm: quadopt} is given in
Appendix~\ref{app:spectral-contrast-proof-v4}.

\subsection{Model-wise recovery and the main theorem}

We first define the scale \(\psi_{\min}\) appearing in the introductory
recovery bound \eqref{eq:intro-recovery-bound-v4}.  For every
\(m\in\mathbb R^{n\times q}\), let
\[
  \psi_{\min}(m)
  :=
  \left(
  \frac{\lambda_{\min}(m^\top m)}
       {\lambda_{\min}(m^\top m)+\delta^2}
  \right)^{1/2},
  \qquad
  \psi_{\min}:=\psi_{\min}(M).
\]
Assumption~\ref{asm:asymp} gives, for every fixed \(\scrO\),
\[
  \lim\limits_{p\uparrow\infty}p^{-1}\Y_{p,\scrO}^\top\Y_{p,\scrO}
  =MM^\top+\delta^2I_n
  \qquad\text{almost surely}.
\]
Since the positive continuous density of \(M\) and \(n>q\) imply
\(\operatorname{rank}(M)=q\) almost surely, the spectral contrast limit yields
\begin{equation}
  \lim\limits_{p\uparrow\infty}
  \lambda_{\min}(\Psi_{p,\scrO})^{1/2}
  =
  \psi_{\min}
  =
  \left(
  \frac{\lambda_{\min}(M^\top M)}
       {\lambda_{\min}(M^\top M)+\delta^2}
  \right)^{1/2}
  >0
  \qquad\text{almost surely}.
  \label{eq:psi-min-v4}
\end{equation}
Equation~\eqref{eq:psi-min-v4} shows that \(\psi_{\min}\) is common to every
orbit-indexed model path. Let \((f_p)_{p\geq q}\) be a sequence of nonrandom measurable estimators
with
\[
  f_p:\mathbb R^{p\times n}\to\mathbb R^{q\times q}
\]
for \(p\geq q\).  For a fixed orbit point \(\scrO\in O(q)\) and
\(\epsilon\geq0\), define the model-specific event
\begin{align}
A_{f,\epsilon}(\scrO)
:=
\Bigl\{
\limsup\limits_{p\uparrow\infty}
\left\|
\Xi_{p,\scrO}-f_p(\Y_{p,\scrO})
\right\|_2
\leq
\epsilon\psi_{\min}
\Bigr\}.
\label{eq:AfO-v4}
\end{align}
We write \(A_f(\scrO):=A_{f,0}(\scrO)\) for exact path recovery under that
model.  The quantity
\(\mathbb P(A_{f,\epsilon}(\scrO))\) is the ordinary success
probability when the parameter \(\scrO\) is fixed.

\begin{theorem}[Main approximate path-recovery bound]\label{thm: approximate-recovery-bound}
Assume \(q\geq2\) and \(0<\epsilon<1/\sqrt2\).  Let
\begin{equation}
  N(\epsilon)
  :=
  \left\lceil\frac{\pi}{2\arcsin(\epsilon)}\right\rceil-1.
  \label{eq:packing-size-v4}
\end{equation}
There exists a set \(\mathbb G_\epsilon\subset O(q)\) with
\(|\mathbb G_\epsilon|=N(\epsilon)\) such that, for every sequence
\((f_p)_{p\geq q}\) of measurable estimators
\(f_p:\mathbb R^{p\times n}\to\mathbb R^{q\times q}\),
\begin{equation}
  \inf_{\scrO\in\mathbb G_\epsilon}
  \mathbb P\bigl(A_{f,\epsilon}(\scrO)\bigr)
  \leq
  u(\epsilon),
  \label{eq: approx-theorem-bound}
\end{equation}
where
\begin{equation}
u(\epsilon)
:=
\min\left\{1,
\inf_K
\left\{
\mathbb P(M\notin K)
+
\left(
\sum_{\scrO\in\mathbb G_\epsilon}
\inf_{m\in K}
\frac{g(m\scrO^\top)}{g(m)}
\right)^{-1}
\right\}\right\},
\label{eq: u-epsilon}
\end{equation}
where the infimum is taken over all nonempty compact sets
\(K\subset\mathbb R^{n\times q}\).
Moreover, \(\lim\limits_{\epsilon\downarrow0}u(\epsilon)=0\).
\end{theorem}

\begin{remark}[Knowledge of the model family]\label{rem:oracle-knowledge-v4}
Each \(f_p\) may be designed with full knowledge of every component of the
candidate model family except the realized index \(\scrO\).  Nevertheless,
the same estimator must be used for every \(\scrO\in\mathbb G_\epsilon\),
and the theorem shows that the probability of its path-recovery event cannot
be uniformly high over that family.
\end{remark}

\begin{corollary}[No uniformly reliable exact path recovery]\label{thm:anglesgeneralized}
Under the assumptions of Theorem~\ref{thm: approximate-recovery-bound}, for
every \(\epsilon\in(0,1/\sqrt2)\), there exists a finite set
\(\mathbb G_\epsilon\subset O(q)\) with
\(|\mathbb G_\epsilon|=N(\epsilon)\), where \(N(\epsilon)\) is given in
\eqref{eq:packing-size-v4}, such that, for every sequence of measurable
estimators \(f=(f_p)_{p\geq q}\),
\[
  \inf_{\scrO\in\mathbb G_\epsilon}
  \mathbb P\bigl(A_f(\scrO)\bigr)
  \leq
  u(\epsilon),
\]
where \(u(\epsilon)\) is defined in \eqref{eq: u-epsilon} for the set
\(\mathbb G_\epsilon\) supplied by that theorem.  Moreover, there exists a
countable set \(\mathbb G_0\subset O(q)\) such that, for every sequence of
measurable estimators \(f=(f_p)_{p\geq q}\),
\begin{equation}
  \inf_{\scrO\in\mathbb G_0}
  \mathbb P\bigl(A_f(\scrO)\bigr)=0.
  \label{eq:exact-recovery-zero-v4}
\end{equation}
\end{corollary}

\begin{proof}
Exact recovery is the special case \(\epsilon=0\), so
\[
A_f(\scrO)\subset A_{f,\epsilon}(\scrO)
\qquad\text{for every }\scrO.
\]
The first claim follows by taking the set \(\mathbb G_\epsilon\) supplied by
Theorem~\ref{thm: approximate-recovery-bound}.  For each integer \(k\geq2\),
choose one such set \(\mathbb G_{1/k}\), independently of the estimator, and
define the countable set
\[
  \mathbb G_0:=\bigcup_{k\geq2}\mathbb G_{1/k}.
\]
For every estimator sequence \(f\) and every \(k\geq2\), the inclusion
\(\mathbb G_{1/k}\subset\mathbb G_0\) gives
\[
\inf_{\scrO\in\mathbb G_0}
\mathbb P\bigl(A_f(\scrO)\bigr)
\leq
\inf_{\scrO\in\mathbb G_{1/k}}
\mathbb P\bigl(A_f(\scrO)\bigr)
\leq u(1/k).
\]
Since \(\lim\limits_{k\uparrow\infty}u(1/k)=0\), this proves
\eqref{eq:exact-recovery-zero-v4}.
\end{proof}

\begin{corollary}[No uniformly reliable path recovery of population eigenvectors]
\label{cor:population-eigenspace-bound}
Let \((h_p)_{p\geq q}\) be measurable estimators of the selected population
eigenvector matrix, with
\(h_p:\mathbb R^{p\times n}\to\mathbb R^{p\times q}\) for
\(p\geq q\).  Define
\[
A^{\mathcal B}_{h,\epsilon}(\scrO)
:=
\left\{
\limsup\limits_{p\uparrow\infty}
\left\|
\mathcal B_{p,\scrO}-h_p(\Y_{p,\scrO})
\right\|_2
\leq
\epsilon\psi_{\min}
\right\}.
\]
Then, for the set \(\mathbb G_\epsilon\) in
Theorem~\ref{thm: approximate-recovery-bound},
\[
  \inf_{\scrO\in\mathbb G_\epsilon}
  \mathbb P\bigl(A^{\mathcal B}_{h,\epsilon}(\scrO)\bigr)
  \leq u(\epsilon).
\]
With \(\mathbb G_0\) the countable set supplied by
Corollary~\ref{thm:anglesgeneralized}, let
\(A^{\mathcal B}_h(\scrO):=A^{\mathcal B}_{h,0}(\scrO)\).
Then
\[
  \inf_{\scrO\in\mathbb G_0}
  \mathbb P\bigl(A^{\mathcal B}_h(\scrO)\bigr)=0.
\]
\end{corollary}

\begin{proof}
If \(h_p\) estimates \(\mathcal B_{p,\scrO}\), define the estimator
\[
  f_p(y)
  :=
  \nu_{p\times q}(yy^\top/n)^\top h_p(y).
\]
For every fixed \(\scrO\), this estimator targets
\(\Xi_{p,\scrO}=\mathcal H_{p,\scrO}^\top\mathcal B_{p,\scrO}\).  Since
\(\|\mathcal H_{p,\scrO}\|_2=1\),
\[
\left\|
\Xi_{p,\scrO}-f_p(\Y_{p,\scrO})
\right\|_2
\leq
\left\|
\mathcal B_{p,\scrO}-h_p(\Y_{p,\scrO})
\right\|_2.
\]
Thus \(A^{\mathcal B}_{h,\epsilon}(\scrO)\subset A_{f,\epsilon}(\scrO)\) for every
\(\scrO\).  Taking the infimum over each orbit and applying
Theorem~\ref{thm: approximate-recovery-bound} and
Corollary~\ref{thm:anglesgeneralized} proves both claims.
\end{proof}

\begin{remark}[Beyond inconsistency of sample PCA]
\label{rem:population-recovery-significance-v4}
Corollary~\ref{cor:population-eigenspace-bound} concerns every measurable
estimator of the selected population eigenvector matrix
\(\mathcal B_{p,\scrO}\), not merely sample PCA.  For each estimator
sequence, the probability of asymptotically exact recovery has infimum zero
over \(\mathbb G_0\), ruling out almost-sure consistency throughout that
family.  This conclusion persists when every component of the model family
except \(\scrO\) is known.  Thus replacing PCA with another
measurable procedure cannot guarantee almost-sure recovery throughout the
stated fixed-sample family.
\end{remark}

\subsection{Why the proof uses fixed sample size and multiple spikes}

The proof uses two structural facts.  First, \(n\) is fixed.  Therefore
\(M=X\scrW\Lambda\) lives in the fixed finite-dimensional space
\(\mathbb R^{n\times q}\), and the likelihood ratios
\[
  \frac{g(m\scrO^\top)}{g(m)}
\]
compare densities on one common space.  If \(n=n_p\) with
\(\lim\limits_{p\uparrow\infty}n_p=\infty\), the same
compact truncation and orbit-summation argument would no longer give a
dimension-independent comparison.
The theorem therefore makes no claim about other possible obstructions when
the sample size grows.

Second, the requirement \(q\geq2\) distinguishes the present negative result
from an established positive result for a single spike.  Under their
single-factor assumptions,
\citet[Theorem~3.1]{goldberg2022dispersion} establish that an observable
spectral contrast estimates \(\langle h,b\rangle^2\) almost surely
consistently, where \(h\) is the leading sample eigenvector and \(b\) its
population counterpart.  Their sign convention then allows its positive
square root to estimate \(\langle h,b\rangle\) almost surely consistently as
the dimension increases with the sample size fixed.  Thus the scalar alignment
can be recovered in that setting after resolving the sign.  Our theorem
establishes an obstruction to uniformly reliable almost-sure path recovery of
the full alignment matrix in the multi-spike setting \(q\geq2\).

\subsection{Proof idea}

Under model \(\scrO\), the score realization \(m\scrO\) produces
\[
  \scrB_p\scrO(m\scrO)^\top+e_p=\scrB_pm^\top+e_p.
\]
Thus different orientations can give the same observations, and hence the
same sample eigenvectors and estimator output.  The orientations in
\(\mathbb G_\epsilon\) are chosen so that their alignment targets are
sufficiently separated to make simultaneous recovery impossible.

This yields pairwise disjoint Borel sets \(F_{K,\scrO}\) representing
successful recovery after the score rotations.  A density comparison bounds
their probabilities from below:
\[
  \mu(F_{K,\scrO})
  \geq c_{\scrO}(K)\left[
  \inf_{\scrO'\in\mathbb G_\epsilon}
  \mathbb P\bigl(A_{f,\epsilon}(\scrO')\bigr)
  -\mathbb P(M\notin K)\right],
\]
where \(K\) is a nonempty compact score set, \(\mu\) is the law of
\((M,\E_\infty)\), and
\[
  c_{\scrO}(K):=\inf_{m\in K}\frac{g(m\scrO^\top)}{g(m)}>0.
\]
The argument now reduces to counting probability mass: disjointness forces
the sum of these lower bounds to be at most one, giving the claimed
restriction on recovery.

\subsection{Conclusion}

We have shown that, in the HDLSS spiked covariance regime with \(q\geq2\),
no estimator of the full sample-population eigenvector alignment matrix
achieves uniformly reliable almost-sure approximate path recovery in the
fixed-sample experiment: under at least one orbit point in a finite
orthogonal packing, the recovery event has probability at most
\(u(\epsilon)\).  The obstruction is not a failure of all spectral
information.  At every fixed orbit point, the same data-only spectral contrast
is almost surely consistent as \(p\uparrow\infty\) for the alignment Gram
matrix \(\Xi_{p,\scrO}\Xi_{p,\scrO}^{\top}\).  By contrast, no single
measurable estimator achieves uniformly reliable path recovery of the full
matrix of pairwise alignments between the selected sample and population
eigenvectors over the packed orbit.  Thus the paper delineates a boundary in
this HDLSS experiment: the Gram summary of eigenvector
alignment is recoverable model by model along the path, but the finer ordered
alignment matrix does not admit uniformly reliable almost-sure path recovery
without additional structure.

\section*{Acknowledgements}
We thank Alec Kercheval for fruitful discussions.  We acknowledge financial
support from the UCLA Olga Radko Endowed Math Circle (ORMC).

\section*{Disclosure of AI assistance}
During preparation of this manuscript, the authors used ChatGPT (OpenAI) for
language editing, structural organization, and consistency checks.  All
mathematical ideas, statements, and proofs were developed and verified by the
authors, who take full responsibility for the manuscript.

\appendix

\section{Proofs and auxiliary lemmas}
\label{sec:appendix}
\subsection{Proof of Theorem~\ref{thm: quadopt}}
\label{app:spectral-contrast-proof-v4}

\begin{proof}
Fix \(\scrO\in O(q)\) and set
\[
  Q_{p,\scrO}:=p^{-1/2}\scrB_p\scrO,
  \qquad
  P_{p,\scrO}:=Q_{p,\scrO}Q_{p,\scrO}^{\top}.
\]
Then \(Q_{p,\scrO}^{\top}Q_{p,\scrO}=I_q\), and the leading population
eigenspace is the column space of \(Q_{p,\scrO}\).  Consequently,
\[
  Q_{p,\scrO}Q_{p,\scrO}^{\top}
  =P_{p,\scrO},
  \qquad
  \Xi_{p,\scrO}\Xi_{p,\scrO}^{\top}
  =\mathcal H_{p,\scrO}^{\top}
   P_{p,\scrO}\mathcal H_{p,\scrO}.
\]
Also,
\[
  \Y_{p,\scrO}
  =\sqrt p\,Q_{p,\scrO}M^\top+\E_p.
\]

We first record the dual Gram limit.  Expanding
\[
  K_{p,\scrO}:=p^{-1}\Y_{p,\scrO}^{\top}\Y_{p,\scrO}
\]
and using \(Q_{p,\scrO}^{\top}Q_{p,\scrO}=I_q\) gives
\begin{align}
  K_{p,\scrO}
  ={}&
  MM^\top
  +
  M\frac{Q_{p,\scrO}^{\top}\E_p}{\sqrt p}
  +\frac{\E_p^\top Q_{p,\scrO}}{\sqrt p}M^\top
  +\frac{\E_p^\top\E_p}{p}.
  \label{eq:dual-gram-expansion-v4}
\end{align}
Assumption~\ref{asm:asymp}(c) implies
\begin{equation}
  \lim\limits_{p\uparrow\infty}
  \frac{Q_{p,\scrO}^{\top}\E_p}{\sqrt p}
  =
  \scrO^\top\lim\limits_{p\uparrow\infty}
  \frac{\scrB_p^\top\E_p}{p}
  =0
  \qquad\text{almost surely}.
  \label{eq:projected-noise-limit-v4}
\end{equation}
Therefore,
\begin{equation}
  \lim\limits_{p\uparrow\infty}K_{p,\scrO}
  =MM^\top+\delta^2I_n
  \qquad\text{almost surely}.
  \label{eq:dual-gram-limit-v4}
\end{equation}
The positive continuous density of \(M\) implies
\(\operatorname{rank}(M)=q\) almost surely.  Hence the limit in
\eqref{eq:dual-gram-limit-v4} is positive definite, so
\(\operatorname{rank}(\Y_{p,\scrO})=n\) for all sufficiently large \(p\)
almost surely.

On this eventual event, let
\[
  d_{i,p,\scrO}:=\frac{n s_{i,p,\scrO}^{\,2}}{p},
  \qquad
  D_{p,\scrO}:=
  \diag(d_{1,p,\scrO},\ldots,d_{q,p,\scrO}),
\]
where \(s_{i,p,\scrO}^{\,2}\) is the \(i\)-th largest nonzero eigenvalue of
\(S_{p,\scrO}\).  Choose the \(n\times q\) matrix \(V_{p,\scrO}\) of
corresponding right singular vectors so that
\begin{equation}
  \mathcal H_{p,\scrO}
  =
  \frac1{\sqrt p}\Y_{p,\scrO}
  V_{p,\scrO}D_{p,\scrO}^{-1/2}.
  \label{eq:left-right-singular-v4}
\end{equation}
The columns of \(V_{p,\scrO}\) are orthonormal and
\[
  V_{p,\scrO}^{\top}K_{p,\scrO}V_{p,\scrO}
  =D_{p,\scrO}.
\]

The key step is an exact finite-\(p\) decomposition.  Since
\((I_p-P_{p,\scrO})Q_{p,\scrO}=0\),
\[
  (I_p-P_{p,\scrO})\Y_{p,\scrO}
  =(I_p-P_{p,\scrO})\E_p,
\]
and therefore
\begin{equation}
  \frac1p\Y_{p,\scrO}^{\top}P_{p,\scrO}\Y_{p,\scrO}
  =
  K_{p,\scrO}
  -\frac1p\E_p^\top\E_p
  +\frac1p\E_p^\top P_{p,\scrO}\E_p.
  \label{eq:projected-dual-identity-v4}
\end{equation}
Combining \eqref{eq:left-right-singular-v4} and
\eqref{eq:projected-dual-identity-v4} yields
\begin{align}
  \Xi_{p,\scrO}\Xi_{p,\scrO}^{\top}
  ={}&
  I_q-
  D_{p,\scrO}^{-1/2}V_{p,\scrO}^{\top}
  \left(
    \frac1p\E_p^\top\E_p
    -\frac1p\E_p^\top P_{p,\scrO}\E_p
  \right)
  V_{p,\scrO}D_{p,\scrO}^{-1/2}.
  \label{eq:alignment-gram-exact-v4}
\end{align}

Define the observable bulk average on the same eventual event by
\[
  \widehat\delta_{p,\scrO}^{\,2}
  :=
  \frac{n}{p(n-q)}
  \sum_{j=q+1}^{n}s_{j,p,\scrO}^{\,2}.
\]
Since \(n_+=n\) eventually, the spectral-contrast definition in
\eqref{eq:def-spectral-contrast-v4}
and the definition of \(D_{p,\scrO}\) give
\begin{equation}
  \Psi_{p,\scrO}
  =
  I_q-\widehat\delta_{p,\scrO}^{\,2}D_{p,\scrO}^{-1}.
  \label{eq:contrast-dual-form-v4}
\end{equation}
Subtracting \eqref{eq:contrast-dual-form-v4} from
\eqref{eq:alignment-gram-exact-v4}, and adding and subtracting
\(\delta^2I_n\), gives the exact decomposition
\begingroup
\interdisplaylinepenalty=10000
\begin{align}
  \Xi_{p,\scrO}\Xi_{p,\scrO}^{\top}-\Psi_{p,\scrO}
  &=
  D_{p,\scrO}^{-1/2}V_{p,\scrO}^{\top}
  \mathcal R_{p,\scrO}
  V_{p,\scrO}D_{p,\scrO}^{-1/2},
  \label{eq:contrast-error-decomposition-v4}\\
  \mathcal R_{p,\scrO}
  &:=
    -\left(\frac1p\E_p^\top\E_p-\delta^2I_n\right)
  \nonumber\\
  &\quad
    +\frac1p\E_p^\top P_{p,\scrO}\E_p
  \nonumber\\
  &\quad
    +\left(\widehat\delta_{p,\scrO}^{\,2}-\delta^2\right)I_n
  .
\end{align}
\endgroup

It remains to show that every term on the right vanishes while the two
factors \(D_{p,\scrO}^{-1/2}\) remain bounded.  First,
\[
  \lim\limits_{p\uparrow\infty}
  \frac1p\E_p^\top P_{p,\scrO}\E_p
  =
  \lim\limits_{p\uparrow\infty}
  \left(
    \frac{Q_{p,\scrO}^{\top}\E_p}{\sqrt p}
  \right)^\top
  \left(
    \frac{Q_{p,\scrO}^{\top}\E_p}{\sqrt p}
  \right)
  \overset{\eqref{eq:projected-noise-limit-v4}}{=}0
  \qquad\text{almost surely}.
\]
Next, Weyl's inequality applied to \eqref{eq:dual-gram-limit-v4} shows that
the bottom \(n-q\) eigenvalues of \(K_{p,\scrO}\) converge to \(\delta^2\).
Thus
\[
  \lim\limits_{p\uparrow\infty}
  \widehat\delta_{p,\scrO}^{\,2}=\delta^2
  \qquad\text{almost surely}.
\]
The same argument shows
\[
  \lim\limits_{p\uparrow\infty}d_{q,p,\scrO}
  =
  \lambda_{\min}(M^\top M)+\delta^2
  >0
  \qquad\text{almost surely},
\]
so \(\|D_{p,\scrO}^{-1/2}\|_2\) is eventually bounded.  Together with
Assumption~\ref{asm:asymp}(c), the operator norm of the right-hand side of
\eqref{eq:contrast-error-decomposition-v4} therefore converges to zero.
This proves \eqref{eq:gram-estimable-v4}.

Finally, for \(i\leq q\),
\[
  \lim\limits_{p\uparrow\infty}(\Psi_{p,\scrO})_{ii}
  =
  \lim\limits_{p\uparrow\infty}
  \left(1-\frac{\widehat\delta_{p,\scrO}^{\,2}}{d_{i,p,\scrO}}\right)
  =
  \frac{\lambda_i(M^\top M)}
       {\lambda_i(M^\top M)+\delta^2}
  \in(0,1)
\]
almost surely, which proves the last assertion.
\end{proof}

\subsection{Proof of the Direct Angle Packing Lemma}

\begin{lemma}[Direct angle packing]\label{lem: v2-direct-angle-packing}
Let \(0<\epsilon<1/\sqrt2\), set
\[
  \alpha_\epsilon:=2\arcsin(\epsilon),
  \qquad
  N(\epsilon):=
  \left\lceil\frac{\pi}{2\arcsin(\epsilon)}\right\rceil-1 .
\]
There exist angles
\[
  \Gamma_\epsilon=
  \{\gamma_0,\gamma_1,\ldots,\gamma_{N(\epsilon)-1}\}
  \subset[0,\pi)
\]
such that, for every \(a\neq b\),
\[
\left|\sin\left(\frac{\gamma_a-\gamma_b}{2}\right)\right|>\epsilon,
\qquad
\left|\cos\left(\frac{\gamma_a-\gamma_b}{2}\right)\right|>\epsilon .
\]
\end{lemma}

\begin{proof}[Proof of Lemma~\ref{lem: v2-direct-angle-packing}]
Let $k:=N(\epsilon)$ and set
\[
\phi\in
\left(
\alpha_\epsilon,\,
\frac{\pi-\alpha_\epsilon}{k-1}
\right).
\]
Such a choice is possible because the definition of $N(\epsilon)$ gives
$k\alpha_\epsilon<\pi$.
Define
\[
\gamma_i:=i\phi,\qquad i=0,1,\ldots,k-1,
\qquad
\Gamma_\epsilon:=\{\gamma_0,\gamma_1,\ldots,\gamma_{k-1}\}.
\]
Then $\Gamma_\epsilon\subset[0,\pi)$. If $a\neq b$, then
\[
\frac{|\gamma_a-\gamma_b|}{2}
=
\frac{|a-b|\phi}{2}
\in\left(\frac{\alpha_\epsilon}{2},
\frac{\pi}{2}-\frac{\alpha_\epsilon}{2}\right).
\]
Therefore
\[
\left|\sin\left(\frac{\gamma_a-\gamma_b}{2}\right)\right|
>
\sin\left(\frac{\alpha_\epsilon}{2}\right)
=
\epsilon,
\quad
\left|\cos\left(\frac{\gamma_a-\gamma_b}{2}\right)\right|
>
\cos\left(\frac{\pi}{2}-\frac{\alpha_\epsilon}{2}\right)
=
\epsilon.
\]
This proves the claim.
\end{proof}

\subsection{Measurable eigenvector selection}

The proof treats estimators as measurable functions of the data, so we fix a
single measurable convention for eigenvectors and singular vectors.  For
$A\in\mathbb R^{m\times \ell}$, let
$\lambda_1(A)>\cdots>\lambda_d(A)$ be the distinct eigenvalue values of $AA^\top$,
including zero when present, where $d=d(A)$.  The associated spectral
projectors are given by the Lagrange interpolation formula
\citep[Section~7.3, p.~529, equation~(7.3.11)]{meyer2000matrix}:
\[
  P_j(A)=\prod_{\substack{1\leq k\leq d\\k\ne j}}
  \frac{AA^\top-\lambda_k(A)I_m}{\lambda_j(A)-\lambda_k(A)},
  \qquad j=1,\ldots,d,
\]
with an empty product interpreted as $I_m$.  The ordered eigenvalues counted
with multiplicity are continuous functions of $A$.  Matrix space therefore
has a finite Borel partition according to eigenvalue multiplicities.  On each
part, the distinct eigenvalue values are continuous and the denominators in
the displayed formula are nonzero, so each projector is continuous there.
Consequently, each $P_j$ is Borel measurable on its domain of definition. Let $\mathcal K_j=\operatorname{range}P_j(A)$, $j=1,\ldots,d$, and set
\[
  r_0=0,\qquad r_j=\sum_{a=1}^{j}\dim(\mathcal K_a).
\]
These mutually orthogonal subspaces exhaust $\mathbb R^m$, so $r_d=m$.
Having selected $v_1,\ldots,v_{k-1}$, select $v_k$ as follows.
\begin{enumerate}
\item Identify the unique $j\in\{1,\ldots,d\}$ satisfying
\[
  r_{j-1}<k\leq r_j.
\]
\item Within $\mathcal K_j$, take the orthogonal complement of the vectors
already selected from that eigenspace:
\[
  \mathcal K_{j,k}
  =\mathcal K_j\cap
  \operatorname{span}\{v_{r_{j-1}+1},\ldots,v_{k-1}\}^{\perp}.
\]
Its orthogonal projector is
\[
  R_{j,k}=P_j(A)-\sum_{a=r_{j-1}+1}^{k-1}v_av_a^\top,
\]
with an empty span interpreted as $\{0\}$ and an empty sum as zero.
Among the standard basis vectors $e_1,\ldots,e_m$, choose the first whose
projection onto $\mathcal K_{j,k}$ is nonzero:
\[
  s_k=\min\{s\in\{1,\ldots,m\}:R_{j,k}e_s\ne0\}.
\]
This index exists because $\dim(\mathcal K_{j,k})=r_j-k+1>0$.
\item Normalize that projection:
\[
  v_k=\frac{R_{j,k}e_{s_k}}{\|R_{j,k}e_{s_k}\|_2}.
\]
\end{enumerate}
Applying this procedure for $k=1,\ldots,m$ produces an orthonormal basis of
left singular vectors, ordered by nonincreasing singular value.  For
$q\leq m$, define
\[
  \nu_{m\times q}(A):=(v_1,\ldots,v_q).
\]
The spectral projectors and their ranks are Borel measurable.  Each
subsequent index is selected by finitely many Borel conditions, and
normalization occurs only at nonzero vectors.  Thus the construction defines
a Borel measurable selector, including at repeated eigenvalues and when
$\operatorname{rank}(A)<q$. For a symmetric positive semidefinite matrix, the same convention is applied to its leading eigenvectors. The precise convention at multiplicities is
irrelevant to the statistical conclusion; it is fixed only so that every
alignment matrix and every recovery event is measurable.

\subsection{Proof of Theorem~\ref{thm: approximate-recovery-bound}}
\label{app: approximate-recovery-proof-v4}

\begin{proof}
Let
\[
R(\gamma)
:=
\begin{pmatrix}
\cos\gamma&-\sin\gamma\\
\sin\gamma&\cos\gamma
\end{pmatrix}.
\]
By Lemma~\ref{lem: v2-direct-angle-packing}, there are angles
\[
\Gamma_\epsilon
=
\{\gamma_0,\ldots,\gamma_{N(\epsilon)-1}\}
\subset[0,\pi)
\]
such that, for every distinct \(\gamma_a,\gamma_b\in\Gamma_\epsilon\),
\[
\left|\sin\left(\frac{\gamma_a-\gamma_b}{2}\right)\right|>\epsilon,
\qquad
\left|\cos\left(\frac{\gamma_a-\gamma_b}{2}\right)\right|>\epsilon.
\]
For \(\gamma\in\Gamma_\epsilon\), define
\[
\scrO(\gamma):=R(\gamma)\oplus I_{q-2},
\qquad
\mathbb G_\epsilon
:=
\{\scrO(\gamma):\gamma\in\Gamma_\epsilon\}.
\]
Then \(|\mathbb G_\epsilon|=N(\epsilon)\).

\paragraph{Model-specific good sets.}
Let \(\mathcal Z:=\mathbb R^{n\times q}\times
\mathbb R^{\mathbb Z_+\times n}\) be equipped with its product Borel
sigma-field.  Let \(\mu\) denote the law on \(\mathcal Z\) of the
coordinate pair \((M,\E_\infty)\).  Let \(\lambda\) be Lebesgue measure
on \(\mathbb R^{n\times q}\), and let \(\mu_\infty\) be the law of
\(\E_\infty\).  Since \(M\) has density \(g\) and is independent of
\(\E_\infty\), we have
\[
\mu(d m,d e_\infty)
=
g(m)\,d\lambda(m)\,d\mu_\infty(e_\infty).
\]

For \((m,e_\infty)\in\mathcal Z\), let \(e_p\) denote the first \(p\)
rows of \(e_\infty\), and write
\[
  y_{p,\scrO}:=\scrB_p\scrO m^\top+e_p.
\]
Here \(\psi_{\min}(m)\) denotes the expression for \(\psi_{\min}\) in
\eqref{eq:psi-min-v4}, evaluated at \(M=m\).  For each
\(\scrO\in\mathbb G_\epsilon\), set
\[
G_{\scrO}:=
\left\{(m,e_\infty)\in\mathcal Z:
\begin{aligned}
&\limsup\limits_{p\uparrow\infty}
\left\|
\nu_{p\times q}\!\left(\frac{y_{p,\scrO}y_{p,\scrO}^{\top}}{n}\right)^\top
\mathcal B_{p,\scrO}-f_p(y_{p,\scrO})
\right\|_2
\\
&\hspace{2em}\leq\epsilon\psi_{\min}(m)
\end{aligned}
\right\}.
\]
The selector and the estimators are Borel measurable, while
\(y_{p,\scrO}\) and \(\psi_{\min}(m)\) depend continuously on the
coordinates.  Consequently, the limsup in the definition of \(G_{\scrO}\)
is Borel measurable, and hence \(G_{\scrO}\) is a Borel subset of
\(\mathcal Z\).
Since \(\mu\) is the law of \((M,\E_\infty)\),
\begin{equation}
\mu(G_{\scrO})
=
\mathbb P\bigl(A_{f,\epsilon}(\scrO)\bigr).
\label{app:eq:model-good-prob-v4}
\end{equation}

For each fixed \(\scrO\), Weyl's eigenvalue perturbation inequality gives
\[
\left|
\sigma_{\min}(\Xi_{p,\scrO})^2
-\lambda_{\min}(\Psi_{p,\scrO})
\right|
\leq
\|\Xi_{p,\scrO}\Xi_{p,\scrO}^{\top}-\Psi_{p,\scrO}\|_2.
\]
By Theorem~\ref{thm: quadopt}, the right-hand side has almost-sure limit
zero.  Equation~\eqref{eq:psi-min-v4} and continuity of the square root
therefore imply
\begin{equation}
\lim\limits_{p\uparrow\infty}\sigma_{\min}(\Xi_{p,\scrO})
=
\lim\limits_{p\uparrow\infty}
\lambda_{\min}(\Psi_{p,\scrO})^{1/2}
=
\psi_{\min}
>0
\qquad\text{almost surely}.
\label{eq:xi-smallest-limit-v4}
\end{equation}
For each \(\scrO\), choose a Borel set \(Z_{\scrO}\subset\mathcal Z\) of
\(\mu\)-probability one on which the spectral-contrast convergence and
\eqref{eq:xi-smallest-limit-v4} hold for that model.  Fix a nonempty compact
\(K\subset\mathbb R^{n\times q}\) and define
\[
H_{K,\scrO}
:=
G_{\scrO}\cap Z_{\scrO}
\cap
\bigl(K\times\mathbb R^{\mathbb Z_+\times n}\bigr).
\]
For every fixed \(\scrO\in\mathbb G_\epsilon\),
\begin{align}
\inf_{\scrO'\in\mathbb G_\epsilon}
\mathbb P\bigl(A_{f,\epsilon}(\scrO')\bigr)
&\leq\mathbb P\bigl(A_{f,\epsilon}(\scrO)\bigr)
=\mu(G_{\scrO})=\mu(G_{\scrO}\cap Z_{\scrO})
\nonumber\\*
&=\mu(H_{K,\scrO})
+\mu\!\left(G_{\scrO}\cap Z_{\scrO}
\cap\bigl(K^c\times\mathbb R^{\mathbb Z_+\times n}\bigr)\right)
\nonumber\\*
&\leq\mu(H_{K,\scrO})+\mathbb P(M\notin K).
\label{app:eq:model-truncation-v4}
\end{align}
Here we used \(\mu(Z_{\scrO})=1\), followed by the partition according to
whether the score coordinate belongs to \(K\).

\paragraph{Transformation to common-data coordinates.}
For \(\scrO\in\mathbb G_\epsilon\), define the Borel bijection
\[
T_{\scrO}(m,e_\infty)
:=
(m\scrO^\top,e_\infty)
\]
and the transformed good set
\[
F_{K,\scrO}:=T_{\scrO}(H_{K,\scrO}).
\]
Set
\[
c_{\scrO}(K)
:=
\inf_{m\in K}
\frac{g(m\scrO^\top)}{g(m)}.
\]
The continuity and strict positivity of \(g\), together with compactness of
\(K\), imply \(c_{\scrO}(K)>0\).

For each fixed \(e_\infty\), make the substitution \(u=m\scrO^\top\),
whose inverse is \(m=u\scrO\).  Right multiplication by \(\scrO^\top\)
acts orthogonally on each of the \(n\) rows, so its absolute Jacobian
determinant is \(|\det\scrO|^n=1\).  The noise coordinate remains unchanged,
and
\[
  (u,e_\infty)\in F_{K,\scrO}
  \quad\Longleftrightarrow\quad
  (m,e_\infty)\in H_{K,\scrO}.
\]
Consequently, by the change-of-variables formula,
\begin{align}
\mu(F_{K,\scrO})
&=\int_{F_{K,\scrO}}g(u)\,d\lambda(u)\,d\mu_\infty(e_\infty)
=\int_{H_{K,\scrO}}g(m\scrO^\top)\,d\lambda(m)\,d\mu_\infty(e_\infty)
\nonumber\\
&=\int_{H_{K,\scrO}}
\frac{g(m\scrO^\top)}{g(m)}\,d\mu(m,e_\infty)
\geq c_{\scrO}(K)\,\mu(H_{K,\scrO}).
\label{app:eq:model-change-v4}
\end{align}
The last inequality follows because \(m\in K\) throughout
\(H_{K,\scrO}\).

\paragraph{Pairwise exclusion.}
We claim that the sets
\(\{F_{K,\scrO}:\scrO\in\mathbb G_\epsilon\}\) are pairwise disjoint.
Suppose, to the contrary, that
\[
(m,e_\infty)\in F_{K,\scrO_1}\cap F_{K,\scrO_2}
\]
for distinct \(\scrO_1,\scrO_2\).  Since
\(T_{\scrO}^{-1}(m,e_\infty)=(m\scrO,e_\infty)\), the coordinate points
\[
(m\scrO_1,e_\infty)\in H_{K,\scrO_1},
\qquad
(m\scrO_2,e_\infty)\in H_{K,\scrO_2}
\]
are good for their respective models.  Both points generate the same
observed data path
\begin{equation}
\overline Y_p
:=
\scrB_pm^\top+e_p,
\label{app:eq:common-data-v4}
\end{equation}
because
\[
\scrB_p\scrO_i(m\scrO_i)^\top+e_p
=
\scrB_pm^\top+e_p,
\qquad i=1,2.
\]

Let \(\overline{\mathcal H}_p\) be the selected sample eigenvector matrix
computed from \(\overline Y_p\overline Y_p^\top/n\).  The leading \(q\)
population eigenvalues are simple and \(p^{-1/2}\scrB_p\scrO_i\) has orthonormal
columns.  Hence, for diagonal sign matrices \(D_{i,p}\),
\[
\nu_{p\times q}(\Sigma_{p,\scrO_i})
=
\mathcal B_{p,\scrO_i}
=
p^{-1/2}\scrB_p\scrO_iD_{i,p},
\qquad i=1,2.
\]
Write the two alignment targets at the coordinate points above as
\[
\Xi_p^{(i)}
:=
\overline{\mathcal H}_p^\top
p^{-1/2}\scrB_p\scrO_iD_{i,p},
\qquad i=1,2.
\]
Because the two coordinate points belong to their respective good sets and
the data in \eqref{app:eq:common-data-v4} are identical,
\begin{align}
\limsup\limits_{p\uparrow\infty}
\|\Xi_p^{(1)}-f_p(\overline Y_p)\|_2
&\leq\epsilon\psi_*,
\\
\limsup\limits_{p\uparrow\infty}
\|\Xi_p^{(2)}-f_p(\overline Y_p)\|_2
&\leq\epsilon\psi_*,
\end{align}
where the two scales agree:
\[
\psi_*
=
\left(
\frac{\lambda_{\min}((m\scrO_i)^\top(m\scrO_i))}
{\lambda_{\min}((m\scrO_i)^\top(m\scrO_i))+\delta^2}
\right)^{1/2}
=
\left(
\frac{\lambda_{\min}(m^\top m)}
{\lambda_{\min}(m^\top m)+\delta^2}
\right)^{1/2}
>0.
\]
The triangle inequality therefore gives
\begin{equation}
\limsup\limits_{p\uparrow\infty}
\|\Xi_p^{(1)}-\Xi_p^{(2)}\|_2
\leq2\epsilon\psi_*.
\label{app:eq:model-pairwise-upper-v4}
\end{equation}
On the other hand, since \(D_{1,p}\) is a diagonal sign matrix,
\(D_{1,p}^{-1}=D_{1,p}\).  Hence
\[
\Xi_p^{(2)}
=
\Xi_p^{(1)}
\left(D_{1,p}\scrO_1^\top\scrO_2D_{2,p}\right),
\]
and hence
\[
\Xi_p^{(1)}-\Xi_p^{(2)}
=
\Xi_p^{(1)}
\left(I_q-D_{1,p}\scrO_1^\top\scrO_2D_{2,p}\right).
\]
Since the first coordinate point lies in \(Z_{\scrO_1}\),
equation~\eqref{eq:xi-smallest-limit-v4} gives
\[
\lim\limits_{p\uparrow\infty}\sigma_{\min}(\Xi_p^{(1)})=\psi_*.
\]
Therefore
\begin{align}
\liminf\limits_{p\uparrow\infty}
\|\Xi_p^{(1)}-\Xi_p^{(2)}\|_2
&\geq
\psi_*
\liminf\limits_{p\uparrow\infty}
\left\|
I_q-D_{1,p}\scrO_1^\top\scrO_2D_{2,p}
\right\|_2.
\label{app:eq:model-pairwise-lower-v4}
\end{align}
Write \(\scrO_i=R(\gamma_i)\oplus I_{q-2}\), set
\(\theta=\gamma_2-\gamma_1\), and let
\(\widetilde D_{i,p}\) be the leading \(2\times2\) block of \(D_{i,p}\).
The leading block of
\(D_{1,p}\scrO_1^\top\scrO_2D_{2,p}\) is
\[
\widetilde D_{1,p}R(\theta)\widetilde D_{2,p}.
\]
If this real orthogonal block has determinant \(-1\), its eigenvalues are
\(1\) and \(-1\), so its distance from \(I_2\) in operator norm is \(2\).
If its determinant is \(1\), the two sign matrices have the same determinant.
When both are \(\pm I_2\), the block is \(\pm R(\theta)\); when each has
one positive and one negative diagonal entry, the block is
\(\pm R(-\theta)\).  Thus it belongs to
\(\{R(\theta),R(-\theta),-R(\theta),-R(-\theta)\}\), whose distances from
\(I_2\) in operator norm are either \(2|\sin(\theta/2)|\) or
\(2|\cos(\theta/2)|\).  Taking the leading \(2\times2\) block therefore
gives, uniformly over the diagonal sign matrices,
\begin{align*}
\left\|
I_q-D_{1,p}\scrO_1^\top\scrO_2D_{2,p}
\right\|_2
&\geq
\left\|I_2-\widetilde D_{1,p}R(\theta)\widetilde D_{2,p}\right\|_2
\\
&\geq
2\min\left\{
\left|\sin\left(\frac{\theta}{2}\right)\right|,
\left|\cos\left(\frac{\theta}{2}\right)\right|
\right\}
>2\epsilon.
\end{align*}
The strict inequality follows from the angle-packing lemma.
The lower bound in \eqref{app:eq:model-pairwise-lower-v4} is therefore
strictly larger than \(2\epsilon\psi_*\), contradicting
\eqref{app:eq:model-pairwise-upper-v4}.  This proves the pairwise
disjointness.

\paragraph{Worst-model probability bound.}
For every nonempty compact \(K\subset\mathbb R^{n\times q}\), pairwise
disjointness and the preceding bounds give
\begin{align*}
1
&\geq\sum_{\scrO\in\mathbb G_\epsilon}\mu(F_{K,\scrO})
\overset{\eqref{app:eq:model-change-v4}}{\geq}
\sum_{\scrO\in\mathbb G_\epsilon}c_{\scrO}(K)\,\mu(H_{K,\scrO})
\\*
&\geq
\left(
\inf_{\scrO'\in\mathbb G_\epsilon}
\mathbb P\bigl(A_{f,\epsilon}(\scrO')\bigr)
-\mathbb P(M\notin K)
\right)
\sum_{\scrO\in\mathbb G_\epsilon}c_{\scrO}(K).
\end{align*}
Since the sum of the \(c_{\scrO}(K)\) is strictly positive, rearranging yields
\[
\inf_{\scrO\in\mathbb G_\epsilon}
\mathbb P\bigl(A_{f,\epsilon}(\scrO)\bigr)
\leq\mathbb P(M\notin K)
+\left(\sum_{\scrO\in\mathbb G_\epsilon}c_{\scrO}(K)\right)^{-1}.
\]
Taking the infimum over \(K\), and noting that the probability on the left
is at most \(1\), proves \eqref{eq: approx-theorem-bound} with
\(u(\epsilon)\) as defined in \eqref{eq: u-epsilon}.

\paragraph{Limit of the bound.}
For \(r>0\), let
\(K_r:=\{m\in\mathbb R^{n\times q}:\|m\|_2\leq r\}\).
For every \(\scrO\in O(q)\),
\[
c_{\scrO}(K_r)
=\inf_{m\in K_r}\frac{g(m\scrO^\top)}{g(m)}
\geq
\frac{\displaystyle\min_{m\in K_r}g(m\scrO^\top)}
{\displaystyle\max_{m\in K_r}g(m)}.
\]
Since \(\|m\scrO^\top\|_2=\|m\|_2\), right multiplication by
\(\scrO^\top\) maps \(K_r\) bijectively onto itself.  Therefore,
\[
\min_{m\in K_r}g(m\scrO^\top)=\min_{m\in K_r}g(m).
\]
Writing
\[
g_r^-:=\min_{m\in K_r}g(m),
\qquad g_r^+:=\max_{m\in K_r}g(m),
\]
continuity and strict positivity of \(g\), together with compactness of
\(K_r\), give \(0<g_r^-\leq g_r^+<\infty\).  Hence
\[
c_{\scrO}(K_r)\geq\frac{g_r^-}{g_r^+}
\qquad\text{for every }\scrO\in O(q).
\]
Using \(K_r\) in the definition of \(u(\epsilon)\) and
\(|\mathbb G_\epsilon|=N(\epsilon)\), we obtain
\begin{align*}
u(\epsilon)
&\leq\mathbb P(\|M\|_2>r)
+\left(\sum_{\scrO\in\mathbb G_\epsilon}c_{\scrO}(K_r)\right)^{-1}
\\
&\leq\mathbb P(\|M\|_2>r)+\frac{g_r^+}{N(\epsilon)g_r^-}.
\end{align*}
For each fixed \(r\), the relation
\(\lim\limits_{\epsilon\downarrow0}N(\epsilon)=\infty\) therefore yields
\[
\limsup\limits_{\epsilon\downarrow0}u(\epsilon)
\leq\mathbb P(\|M\|_2>r).
\]
Finally, \(\lim\limits_{r\uparrow\infty}\mathbb P(\|M\|_2>r)=0\)
and \(u(\epsilon)\geq0\) imply
\(\lim\limits_{\epsilon\downarrow0}u(\epsilon)=0\).
\end{proof}

\newpage
\bibliographystyle{agsm}
\bibliography{quadopt}

\end{document}